\documentclass[a4paper,article,11pt]{article}
\usepackage{amssymb}
\usepackage{amstext}
\usepackage{amsmath}
\usepackage{amscd}
\usepackage{amsthm}
\usepackage{amsfonts}
\usepackage{enumerate}
\usepackage{latexsym}
\usepackage{mathrsfs}
\usepackage{mathtools}
\usepackage{tikz}

\makeatletter

\@addtoreset{equation}{section}
\makeatother

\newtheorem{theorem}{Theorem}[section]
\newtheorem{lemma}[theorem]{Lemma}

\newtheorem{remark}[theorem]{Remark}

\DeclareMathOperator{\wn}{Wind}

\DeclareMathOperator{\supp}{supp}

\DeclareMathOperator{\Index}{Index}

\DeclareMathOperator{\ac}{ac}

\DeclareMathOperator{\pp}{p}

\newcommand{\strong}{\mathop{\rm s\mathchar`-lim}\limits}

\def\C{\mathbb C}
\def\N{\mathbb N}
\def\S{\mathbb S}
\def\R{\mathbb R}
\def\Z{\mathbb Z}
\def\A{\mathcal A}

\def\d{\mathrm d}

\def\e{\mathrm e}
\def\E{\mathcal E}
\def\F{\mathscr F}
\def\H{\mathcal H}

\def\J{\mathcal J}
\def\K{\mathcal K}

\def\U{\mathcal U}
\def\W{\mathcal W}

\def\V{\mathbb V}

\begin{document}
\title{Explicit formula for Schr\"odinger wave operators on the half-line for potentials up to optimal decay}
\author{H. Inoue}

\date{\small}
\maketitle \vspace{-1cm}

\begin{quote}
\emph{
\begin{itemize}
\item[] Graduate school of mathematics, Nagoya University,
Chikusa-ku, Nagoya 464-8602, Japan
\item[] \emph{E-mail:} m16007v@math.nagoya-u.ac.jp
\end{itemize}
}
\end{quote}

\maketitle

\begin{abstract}
We give an explicit formula for the wave operators for Schr\"odinger operators on the half-line with a potential decaying strictly faster than the polynomial of degree minus two. The formula consists of the main term given by the scattering operator and a function of the generator of the dilation group, and a Hilbert-Schmidt remainder term. Our method is based on the elementary construction of the generalized Fourier transforms in terms of the solutions of the Volterra integral equations. As a corollary, a topological interpretation of Levinson's theorem is established via an index theorem approach.
\end{abstract}

\smallskip


\section{Introduction}

Common objects in mathematical scattering theory are the \emph{{\rm (}M\o ller{\rm )} wave operators}. In the time-dependent framework, they are defined by the strong limits
$$
W_{\pm}\equiv W_{\pm}(H,H_0):=\strong_{t\to\pm\infty}\e^{itH}\e^{-it H_0}{\bf 1}_{\ac}(H_0)
$$
for a pair of self-adjoint operators $(H,H_0)$ on a Hilbert space $\H$ whenever the above limits exist. Here, ${\bf 1}_{\ac}(H_0)$ denotes the spectral projection associated with the absolutely continuous subspace $\H_{\ac}(H_0)$ of $H_0$. $W_{\pm}$ are, whenever they exist,  partial isometries with ranges contained in the absolutely continuous subspace $\H_{\ac}(H)$ of $H$, and therefore it is natural to say $W_{\pm}$ are \emph{complete} if their ranges coincide with $\H_{\ac}(H)$. Important consequences of the existence and completeness are that they give rise to unitary equivalences between $\H_{\ac}(H_0)$ and $\H_{\ac}(H)$, and that the \emph{scattering operator} $S:=W_{+}^*W_-$ is a unitary operator on $\H_{\ac}(H_0)$. 


After they were introduced by physicists M\o ller and Heisenberg, numerous works have been devoted to the problem of the existence and completeness. As a leading example, mathematical scattering theory is developed for and efficiently applied to the study of Schr\"odinger operators. We can touch only a few aspects of this theory in this paper and we refer to \cite{BW,Y3} for more information.


A next step in the mathematical scattering theory is to study some structural formulas or further mapping properties of $W_{\pm}$. It is usually based on the stationary approach. We mention that the $L^p$-boundedness of $W_{\pm}$ for a Schr\"odinger operator on $\R^d$ initiated by Yajima \cite{Yajima} goes in this direction. Indeed, the precise high- and low-energy analysis of $W_{\pm}$ in terms of the stationary expression are commonly used. 


In the last 10 years, several new explicit formulas for $W_{\pm}$ have been obtained as for example in \cite{KR,DR} (see also \cite{R} for more information). These works are mostly motivated for giving a topological interpretation as an index theorem to the famous \emph{Levinson's theorem}, which relates the scattering part to the number of bound states of a quantum system. It has been originally established for a Schr\"odinger operator $-\Delta+V(X)$ with a spherically symmetric potential $V$ i.e. $V(\cdot)=v(|\cdot|)$ for some $v:\R_+\to\R$, by N. Levinson.  For the orbital quantum number $\ell=0$, the radial part of the operator reduces to a Schr\"odinger operator of the form $-\d^2/\d x^2+v(X)$ on the half-line $\R_+$ and this relation is often formulated with the \emph{phase shift} $\eta(k)=\eta_{\ell=0}(k)$, i.e. if we set $N$ to be the number of bound states of the system, then the equality
\begin{equation}\label{classical Levinson}
\eta(\infty)-\eta(0)=\pi \left(N+\delta\right)
\end{equation}
holds, where the correction term $\delta=1/2$ if there exists a \emph{zero-energy resonance} (see, Section 2) and $\delta=0$ otherwise. The relation \eqref{classical Levinson} is usually proved based on complex analysis (see \cite[Thm. XI. 59]{RS} or \cite[Thm. 4.6.1]{Y1}), and has been generalized to several quantum systems by many researchers on purely analytical basis. 


In contrast to the usual approach, the topological approach \cite{KR} is based on some algebraic methods including $K$-theory for $C^*$-algebras. Explicit formulas for $W_-$ have been used to prove its affiliation to a $C^*$-subalgebra $\E\subset B(\H)$ with the quotient algebra $\E/K(\H)$ isomorphic to $C(\S^1)$.  A concise introduction to this approach can be found in the review paper \cite{R}. More recently in \cite{IR,NPR}, this approach is generalized to specific potentials of the form $x^{-2}$ on the half-line $\R_+$ with possibly complex coefficients, which create infinitely many negative eigenvalues or finitely many complex eigenvalues.


The aim of the present paper is to provide an explicit formula for $W_-$ and give a topological interpretation to the relation \eqref{classical Levinson} for the Schr\"odinger operator $H=H_0+v(X)$ with $H_0=-\d^2/\d x^2$ being the Dirichlet Laplacian on $\R_+$.  In this paper, the real-valued potential $v\in L^{\infty}(\R_+)$ is assumed to decay strictly faster than $-2$, i.e. there exist $C>0$ and $\rho>2$ such that
\begin{equation}\label{decay assumption}
|v(x)|\leq C(1+x)^{-\rho}, \qquad\forall x\geq0.
\end{equation}
This assumption is strong enough to ensure the existence and completeness of $W_\pm$ for the pair $(H,H_0)$. It is also well known that $H$ has no singular continuous spectrum for such potential $v$ and has finitely many simple negative eigenvalues. Indeed, $\rho=2$ is the borderline case for the finiteness of the number of bound states.


We now set $A$ to be the generator of the dilation group defined by $[\e^{-itA}f](x):=\e^{t/2}f(\e^t x)$ for $f\in L^2(\R_+)$ and $x\in\R_+$, and $S$ to be the scattering operator for $(H,H_0)$, in order to state the following main result:
\begin{theorem}\label{expression of wo}
Under the assumption \eqref{decay assumption} with $\rho>2$, the equality
\begin{equation}\label{formula of the wo}
W_{-}={\bf 1}+\phi(A)(S-{\bf 1})+K
\end{equation}
holds, where $\phi(A):=(i\e^{\pi A}+1)^{-1}$ and $K$ is a Hilbert-Schmidt operator on $L^2(\R_+)$. 
\end{theorem}


Note that a quite similar formula for one-dimensional Schr\"odinger operators has already appeared in \cite{KR} but they could not reach the borderline case. Indeed, the potential $v:\R\to\R$ was assumed to satisfy
\begin{equation}\label{l11}
\int_{\R}(1+|x|)^{\alpha}|v(x)|\d x<\infty
\end{equation}
for some $\alpha>5/2$, and this condition means that $v(x)=o(|x|^{-1-\alpha})$ as $|x|\to\infty$ in the power scale. 


Let us mention that the main difficulty in such a work usually consists in proving that the remainder term is compact. Indeed, the relatively new thing developed in this paper is a new scheme to prove the compactness discussed in Section 5. For our model, the remainder term $K$ is at first defined as a kind of oscillatory integral operator with the kernel in terms of the Jost solution, which is a unique solution of an integral equation (see eq.\eqref{Volterra for Jost}). By iterating that integral equation finitely many times, we reduce this operator to a sum of a good part and bad parts. Then, we obtain a rapid decay of the kernel of the good part due to the iterating process. We can decompose each bad part into a product of a Hilbert-Schmidt operator and a bounded operator by using a little algebraic trick of integration by parts.


Let us now be more precise on the contents of this paper. In Section 2, we review basic spectral results on the Schr\"odinger operator $H$ on the half-line. In particular, we study the Schr\"odinger equation $-u''+vu=k^2 u$ in terms of the Volterra integral equations. Basic properties of two distinguished solutions, called the regular solution $\varphi(\cdot,k)$ and the Jost solution $\theta(\cdot,k)$, and their Wronskian $w(k)$, called the Jost function, are recalled in detail. In order to make our exposition self-contained, we repeat the relevant materials form \cite[Chapter 4]{Y1} without proofs.


We will also borrow the stationary expression of the wave operators in terms of generalized Fourier transforms (see Theorem \ref{wave operator}) from that reference. In Section 3, we deduce the above explicit formula for $W^-$ from that expression. Let us mention that the formula implicitly appeared in the proof of the equality between the time-dependent definition and the stationary expression. Indeed, the decomposition of the kernel of the generalized Fourier transform, which is defined by $kw(k)^{-1}\varphi(x,k)$, is the main step. Our formula is obtained by rewriting this decomposition as an operator identity.


In Section 4, we briefly recall the $C^*$-algebraic framework and formulate the topological version of Levinson's theorem for $H$. Since both analytic and topological indices can be computed explicitly, we will not touch $K$-theory behind the index theorem but more information in this direction can be found in \cite{R}. The interest is pushed forward to the computation of the topological index, which is defined as a winding number of a continuous invertible function on the edge of a square and therefore comes with four contributions, one from each edge. This computation enables us to compare the resulting index theorem with the usual Levinson's theorem \eqref{classical Levinson}.


In Section 5, the proof of the compactness of the remainder term is provided, and as we have already mentioned, this part is the main contribution of this paper.  Based on the classical techniques of integral equations as in \cite{AK,DF,GNP}, we study the integral kernels algebraically by playing with trigonometric functions and integration by parts. For the convenience of the reader, we perform some computations explicitly.


As a final remark, let us mention that in the paper \cite{DF}, the $L^p$-boundedness of the wave operators for one-dimensional Schr\"odinger operators for $1<p<\infty$ is proved under the assumption \eqref{l11} with $\alpha=1$. We will not develop this point in this paper, but it is natural to wonder if we can deduce such property of the wave operators in terms of our formula. 


\subsection*{Acknowledgement}
This work is supported by JSPS KAKENHI Grant number JP18J21491. The author gratefully acknowledges the many helpful suggestions of professor S. Richard during the preparation of this paper, and has also greatly benefited from discussions with N. Tsuzu.


\section{Preliminaries}

Based on \cite[Chapter 4]{Y1}, let us recall some well-known facts about the \emph{regular solution} $\varphi(x,\zeta)$ and the \emph{Jost solution} $\theta(x,\zeta)$ of the Schr\"odinger equation
\begin{equation}\label{schrodinger equation}
-u''(x)+v(x)u(x)=zu(x),\qquad x>0 
\end{equation}
for $z=\zeta^2$ with $\zeta\in\C\setminus\{0\}$. Here, the real-valued potential $v$ satisfies \eqref{decay assumption} with $\rho>2$. The direct approach to scattering theory reviewed in this section can be applied to a more general class of potentials, for example one can admit slower decays at infinity or some singularities at $x=0$. However, we assume \eqref{decay assumption} throughout the present paper to simplify our presentation. Note that one easily infers that the operator $H:=H_{0}+v(X)$ with $H_{0}=-\d^2/\d x^2$ being the Dirichlet Laplacian on $\R_+$, is self-adjoint in $L^2(\R_+)$ with the same domain as $H_0$. 


Recall that the \emph{regular solution} $\varphi(\cdot,\zeta)$ is the solution of \eqref{schrodinger equation} satisfying the following boundary condition at $x=0$:
\begin{equation}\label{regular boundary condition} 
\varphi(0,\zeta)=0\qquad \text{and}\qquad \varphi'(0,\zeta)=1,
\end{equation}
where $'$ denotes the derivative with respect to the variable $x$. In the free case, when the potential $v$ is identically zero, the regular solution is given by $\varphi_0(x,\zeta)=\zeta^{-1}\sin(\zeta x)$. For a function $\varphi(\cdot,\zeta)$ satisfying \eqref{regular boundary condition}, the Schr\"odinger equation is equivalent to the following Volterra integral equation:
\begin{equation}\label{Volterra for regular}
\varphi(x,\zeta)=\varphi_0(x,\zeta)+\frac{1}{\zeta}\int_{0}^x\sin\bigl(\zeta(x-y)\bigr)v(y)\varphi(y,\zeta)\d y,\qquad x\geq0.
\end{equation}
By applying the method of successive approximations, one can prove that there exists a unique solution $\varphi(\cdot,\zeta)$ of \eqref{Volterra for regular} for any $\zeta\in\C$, including the point $\zeta=0$. Note also that for each fixed $x\geq0$ this solution satisfies $\varphi(x,\zeta)=\varphi(x,-\zeta)$, and that the function $\varphi(x,\zeta)$ is an entire function of the variable $z=\zeta^2$ \cite[Lem. 4.1.2]{Y1}.


For $\Im(\zeta)$ with $\zeta\neq0$, the \emph{Jost solution} $\theta(\cdot,\zeta)$ is the solution of \eqref{schrodinger equation} satisfying the following boundary condition at infinity:
\begin{equation}\label{Jost boundary condition} 
\theta(x,\zeta)=\e^{i\zeta x}\bigl(1+o(1)\bigr)\ \ \ \ \text{and}\ \ \ \theta'(x,\zeta)=i\zeta \e^{i\zeta x}\bigl(1+o(1)\bigr)\qquad \text{as}\ x\to\infty.
\end{equation}
Note that $\theta_0(x,\zeta):=\e^{i\zeta x}$ is the Jost solution in the free case, and the corresponding integral equation is
\begin{equation}\label{Volterra for Jost}
\theta(x,\zeta)=\theta_0(x,\zeta)+\frac{1}{\zeta}\int_{x}^{\infty}\sin\bigl(\zeta(x-y)\bigr)v(y)\theta(y,\zeta)\d y,\qquad x\geq0.
\end{equation}
Similarly, one can prove that there exists a unique solution $\theta(\cdot,\zeta)$ of \eqref{Volterra for Jost} for any $\zeta\neq0$ in the closure of the upper half-plane $\C_+:=\{\zeta\in\C\mid \Im(\zeta)>0\}$, and also that for each fixed $x\geq0$, $\theta(x,\zeta)$ is an analytic function of $\zeta$ in the upper-half plane, and continuous up to the real axis except $\zeta=0$ \cite[Lem. 4.1.4]{Y1}. The following estimate is also proved in the same lemma: for any $k_0>0$ there is $C>0$ such that for any $|\zeta|>k_0$ and $x\geq0$
\begin{equation}\label{estimate1}
|p(x,\zeta)|\leq \frac{C\e^{-\Im(\zeta)x}}{|\zeta|}\int_{x}^{\infty}|v(y)|\d y,
\end{equation}
where $p(x,\zeta):=\theta(x,\zeta)-\theta_0(x,\zeta)$.


Since the function $x\mapsto xv(x)$ still belongs to $L^1(\R_+)$, one gets the following estimate for the low-energy \cite[Lem. 4.3.1]{Y1}: for fixed $x\geq0$ the Jost solution $\theta(x,\zeta)$ is continuous as $\zeta\to0$, $\Im(\zeta)\geq0$, and there exists $C>0$ such that for any $x\geq0$ and $\Im(\zeta)\geq 0$
$$
|p(x,\zeta)|\leq \e^{-\Im(\zeta)x}\left\{\exp\left(C\int_{x}^{\infty}y|v(y)|\d y\right)-1\right\}.
$$
One can easily infer from the Taylor expansion of the function $t\mapsto\e^t$ that there exists $C'>0$ such that
\begin{equation}\label{estimate2}
|p(x,\zeta)|\leq C'\int_{x}^{\infty}y|v(y)|\d y.
\end{equation}
Moreover, $\theta(x):=\theta(x,0)$ is a solution of the Schr\"odinger equation
\begin{equation}
-\frac{\d^2}{\d x^2}u(x)+v(x)u(x)=0,\qquad x>0.
\end{equation}


In the following, we assume $\Im(\zeta)\geq0$, and we write $k$ instead of $\zeta$ if $\Im(\zeta)=0$ and $\zeta>0$. The Wronskian $w(\zeta):=\varphi'(x,\zeta)\theta(x,\zeta)-\varphi(x,\zeta)\theta'(x,\zeta)$ is called the \emph{Jost function}. Since it is independent of the variable $x$, it follows from the boundary condition \eqref{regular boundary condition} that $w(\zeta)=\theta(0,\zeta)$. By using the uniqueness of the solutions $\varphi(\cdot,k)$ and $\theta(\cdot,k)$, one has
\begin{equation*}
\theta(x,-k)=\overline{\theta(x,k)}\qquad \text{whence}\qquad w(-k)=\overline{w(k)},
\end{equation*}
and 
\begin{equation}\label{regular by Jost}
\varphi(x,k)=\frac{1}{2ik}\Bigl(\theta(x,k)w(-k)-\theta(x,-k)w(k)\Bigr).
\end{equation}
Therefore, $w(k)\neq0$ for $k>0$ since otherwise it would follow from \eqref{regular by Jost} that $\varphi(\cdot,k)$ is identically $0$. Hence, by continuity of the function $w$, there exists a unique continuous function $\eta:\R_+\to\R$ satisfying $\displaystyle\lim_{k\to+\infty}\eta(k)=0$ such that for any $k>0$
\begin{equation*}
w(k)=A(k)\e^{i\eta(k)}\qquad \text{with}\qquad A(k):=|w(k)|.
\end{equation*}
The coefficient $A(k)$ is called the limit amplitude and $\eta(k)$ is called the \emph{phase shift} associated with $v$. Moreover, the \emph{scattering matrix} is now defined by
\begin{equation}\label{scattering matrix}
s(k):=\frac{w(-k)}{w(k)}=\e^{-2i\eta(k)},\qquad  k>0.
\end{equation}


It is known \cite[Lem. 4.2.2]{Y1} that the Jost function $w:\overline
{\C_{+}}\setminus\{0\}\to\C$ has complex zeros only on $i\R_+$, and they are simple. Moreover, $w(\zeta)=0$ if and only if $\lambda=\zeta^2$ is an eigenvalue of $H$. Indeed, if $w(\zeta)=0$ then the Jost solution $\theta(\cdot,\zeta)$ of $H$ is the eigenfunction with eigenvalue $\lambda=\zeta^2$. It is also well known that $H$ has only finitely many negative eigenvalues (see, \cite[Rem. 3.8, 182p]{Y1}) 


As a consequence, $k=0$ is the only possible real zero of the Jost function $w$. Since $\theta(\cdot,0)$ is not a $L^2$-function but satisfies the Schr\"odinger equation with energy $\lambda=0$, we say that $H$ has a \emph{zero-energy resonance} if $w(0)=0$. It is also called a \emph{half-bound state}  since it contributes to the classical Levinson's theorem \eqref{classical Levinson} as a correction term $\delta$ by $1/2$. Note that the zero-energy resonance also affects on the low-energy asymptotic of the scattering matrix $s(k)$. More precisely, as $k\downarrow 0$
\begin{align*}
s(k)=\begin{cases}
\ \ 1+o(1) & \text{\rm if}\ w(0)\neq 0,\\
-1+o(1) & \text{\rm if}\ w(0)=0,
\end{cases}
\end{align*}
see \cite[Prop. 4.3.4 and Prop. 4.3.9]{Y1}.


\section{Formula for the wave operator $W_-$}

We turn now to the scattering theory for the Schr\"odinger operator $H$. In our situation, $W_{-}$ (and $W_{+}$ resp.) exists, and as we will see below, maps $\psi_0(x,k)=\sin(kx)$ to the wave function
\begin{equation*}
\psi^{-}(x,k):=kw(k)^{-1}\varphi(x,k)\ \ \Bigl(\text{and}\ \psi^{+}(x,k):=\overline{\psi^{-}(x,k)}\ \text{resp.} \Bigr).
\end{equation*}
More precisely, $W_\pm$ map a wave packet of the form
\begin{equation}\label{sine transform}
[\F_{s}g](x):=\sqrt{\frac{2}{\pi}}\int_{0}^{\infty}\sin(kx)g(k)\d k,\qquad \forall x\geq 0
\end{equation}
to the wave packets
\begin{equation}\label{generalized Fourier transform}
[\F^{\pm*}g](x):=\sqrt{\frac{2}{\pi}}\int_{0}^{\infty}\psi^{\pm}(x,k)g(k)\d k,\qquad \forall k\geq 0
\end{equation}
for $g\in C_{\rm c}^{\infty}(\R_+)$.


The equation \eqref{sine transform} defines the \emph{Fourier sine transform} $\F_{s}$, which continuously extends to a self-adjoint unitary operator on $L^2(\R_+)$. It also satisfies the intertwining relation $\F_{s}H_0=L\F_{s}$, where $L$ is the multiplication operator by the function $k\mapsto k^2$. On the other hand, the \emph{generalized Fourier transforms} $\F^{\pm}$ associated with $H$ defined by
\begin{equation*}
[\F^{\pm}f](k):=\sqrt{\frac{2}{\pi}}\int_{0}^{\infty}\overline{\psi^{\pm}(x,k)}f(x)\d x,\qquad \forall f\in C_{\rm c}^{\infty}(\R_+),\ k\geq 0
\end{equation*}
also continuously extend to bounded linear operators on $L^2(\R_+)$, and \eqref{generalized Fourier transform} actually correspond to the adjoints of $\F^{\pm}$. They satisfy $\F^{\pm}\F^{\pm*}={\bf 1}$, $\F^{\pm*}\F^{\pm}={\bf 1}_{\ac}(H)$ and the intertwining relations $\F^{\pm}H=L\F^{\pm}$, where ${\bf 1}_{\ac}(H)={\bf 1}_{\R_+}(H)$ is the projection onto the absolutely continuous subspace $\H_{\ac}(H)$ of $H$.


\begin{theorem}\label{wave operator}
Under \eqref{decay assumption}, the wave operators $W_{\pm}$ for the pair $(H,H_0)$ exist, satisfy $W_{\pm}=\F^{\pm*}\F_{s}$ and are complete.
\end{theorem}


\begin{remark}
The conclusion of Theorem \ref{wave operator} is true for more general class, e.g. $v\in L^1(\R_+)$, see \cite[Thm. 4.2.8 and Rem. 2.9, p.176]{Y1}. 
\end{remark}


We recall the main idea of the proof for $W_-$ but the one for $W_+$ is almost the same. For $f=\F_{s}g$ with $g\in C_{\rm c}^{\infty}(\R_+)$, by virtue of intertwining property one has
\begin{equation*}
\|\e^{itH}\e^{-itH_0}f-\F^{-*}\F_{s}f\|_{L^2(\R_+)}=\|(\F^{-*}-\F_{s})\e^{-itL}g\|_{L^2(\R_+)}.
\end{equation*}
Hence, it suffices to prove that the $L^2$-norm of the function
\begin{equation}\label{integral}
[(\F^{-*}-\F_{s})\e^{-itL}g](x)=\sqrt{\frac{2}{\pi}}\int_0^{\infty}\left(\psi^-(x,k)-\sin(kx)\right)\e^{-itk^2}g(k)\d k
\end{equation}
tends to $0$ as $t\to-\infty$. By using \eqref{regular by Jost} and \eqref{scattering matrix},  one has
\begin{align}\notag
\sqrt{\frac{2}{\pi}}\left(\psi^-(x,k)-\sin(kx)\right)&=\sqrt{\frac{2}{\pi}}\left(\frac{\e^{ikx}}{2i}\bigl(s(k)-1\bigr)\right)+\sqrt{\frac{2}{\pi}}\frac{1}{2i}\left\{p(x,k)s(k)-\overline{p(x,k)}\right\}\\ 
&=:F_1(x,k)+F_2(x,k).\label{kernel}
\end{align}
One can prove that the norm of the term in \eqref{integral} defined by the kernel $F_1$ tends to $0$ as $t\to-\infty$ by the standard argument \cite[Lem. 0.4.9]{Y1} already used in the proof of the invariance principle. By using the estimate \eqref{estimate1}, there exists $C_g>0$ such that for $k\in\supp g$ and $x\in\R_+$,
\begin{equation*}
|F_2(x,k)|\leq C_g\int_{x}^{\infty}|v(y)|\d y.
\end{equation*}
Since $x\mapsto \int_{x}^{\infty}|v(y)|\d y$ belongs to $L^2(\R_+)$, the norm of the term defined by $F_2$ in \eqref{integral} tends to $0$ as $|t|\to\infty$ by the Riemann-Lebesgue Lemma and the Lebesgue dominated convergence theorem. 


It follows from Theorem \ref{wave operator} that the scattering operator $S=W_{+}^*W_-$ is a unitary operator on $L^2(\R_+)$, and since $s\psi^-=\psi^+$, the equality
\begin{equation}\label{stationary scattering matrix}
S=\F_s s\F_s
\end{equation}
holds, where $s$ denotes the multiplication operator by the scattering matrix $k\mapsto s(k)$.


The decomposition \eqref{kernel} is our starting point for an explicit formula of the wave operator. Indeed, if we set $F_{j}$ for the integral operator with the kernel $F_j(x,k)$, $j=1,2$, then
\begin{equation}\label{decomp as operator}
W_-=\F^{-*}\F_s={\bf 1}+F_1\F_s+F_2\F_s.
\end{equation}
By letting $\F_c$ denote the \emph{Fourier cosine transform}, one obtains $F_1\F_s=\frac{1}{2i}(\F_c\F_s+i{\bf 1})(S-{\bf 1})$. Moreover, one can deduce from \cite[eq. (4.29)]{DR} and from some trigonometric identities,
\begin{equation}\label{function of A}
\frac{1}{2i}(\F_c\F_s+i{\bf 1})=\frac{1}{i\e^{\pi A}+1}=:\phi(A),
\end{equation}
where $A$ is the generator of the dilation group on $L^2(\R_+)$, that is, $[\e^{-itA}f](x):=\e^{t/2}f(\e^t x)$ for $f\in L^2(\R_+)$ and $x\in\R_+$. Here, we have changed the sign of the generator of dilation group from that in \cite[p.875]{DR}. Finally, by combining \eqref{stationary scattering matrix} with \eqref{function of A}, we obtain the expression \eqref{expression of wo} from \eqref{decomp as operator} with $K:=F_2\F_{s}$.


\begin{remark}
When we consider $W_+$ in place of $W_-$, the adjoint $S^*$ naturally appears. Similarly, one obtains the expression
$$
W_+={\bf 1}+\psi(A)(S^*-{\bf 1})+F_2'\F_s
$$
where $\psi(A)=(1-i\e^{-\pi A})^{-1}$. The kernel of $F_2'$ is given by $F_{2}(x,k)\overline{s(k)}$ and is therefore a Hilbert-Schmidt operator.  We can also see this from $W_+=W_-S^*$
\end{remark}


We provide a detailed analysis of the remainder term $K$ in Section 5. Note that in the case $\rho>5/2$, the proof of the compactness can be very simplified. Indeed, it follows from \eqref{decay assumption}, \eqref{estimate1} and \eqref{estimate2} that there exists $C>0$ such that for any $x\geq0$
\begin{align}\label{estimate of K2}
|F_2(x,k)|\leq \sqrt{\frac{2}{\pi}}\bigl|p(x,k)\bigr|\leq C\begin{cases}
k^{-1}(1+x)^{-(\rho-1)} & \text{if}\ k>1,\\
(1+x)^{-(\rho-2)} & \text{if}\ 0\leq k\leq 1.
\end{cases}
\end{align}
Here, we used the triangle inequality and that $|s(k)|=1$. Hence, $F_2$ belongs to $L^p(\R_+\times \R_+)$ for $p>1$ if $\rho>2+1/p$. In particular,  it follows that $F_2$ is a Hilbert-Schmidt operator provided $\rho>5/2$, and so is $K$.


\section{Topological version of Levinson's theorem}

In this section, we provide a $C^*$-algebraic framework, which has already appeared in \cite{KR}, and give a topological interpretation of the usual Levinson's theorem \eqref{classical Levinson}. In the following, $C(\overline{\R})$ and $C_0(\R)$ denote the $C^*$-algebras of continuous functions on $\R$ having limits at $\pm\infty$ and of continuous functions vanishing at $\pm\infty$, respectively.


Note first that our assumption on the decay of the potential is sufficient for the absence of singular continuous spectrum for the operator $H$. Together with the finiteness of the number of eigenvalues of $H$, this implies that $W_-$ is a Fredholm operator.


We define the $C^*$-algebra $\E_{\square}$ on $L^2(\R_+)$ by
\begin{equation}\label{C*-algebra E}
\E_{\square}:=C^*\Bigl( a(A)b(B)\ \Bigr|\ a,b\in C(\overline{\R}) \Bigr),
\end{equation}
where $B:=\frac{1}{2}\ln(H_0)$ is a rescaled energy operator. Here, the right hand side of \eqref{C*-algebra E} stands for the $C^*$-algebra generated by operators of the form $a(A)b(B)$ with $a,b\in C(\overline{\R})$. Since the Weyl canonical commutation relation $\e^{itB}\e^{isA}=\e^{-ist}\e^{isA}\e^{itB}$ holds and the multiplicities of the spectra $\sigma(B)$ and $\sigma(A)$ are one,  the pair $(B,A)$ is unitarily equivalent to the pair $(X,D)$ on $L^2(\R)$, where $X$ and $D=-i\d/\d x$ are the position and momentum operators on $\R$, respectively.


As a consequence, the set $\K=K\left(L^2(\R_+)\right)$ of compact operators on $L^2(\R_+)$ is isomorphic to 
\begin{equation*}
C^*\Bigl( a(A)b(B)\ \Bigr|\ a,b\in C_0(\R) \Bigr).
\end{equation*}
One easily observes that $\K$ is an ideal of $\E_{\square}$ and the quotient morphism $\pi_{\square}:\E_{\square}\to\E_{\square}/\K\cong C(\square)$ is given by
\begin{equation*}
\pi_{\square}\bigl(a(A)b(B)\bigr)=\Bigl(a(-\infty)b, b(-\infty)a, a(+\infty)b, b(+\infty)a\Bigr),
\end{equation*}
where the topological space $\square$ is defined to be the boundary of the compact space $\overline{\R}\times\overline{\R}$. 


We define the rescaled scattering matrix $S(\beta):=s(\e^{\beta})$ for $\beta\in\R$. Then, we can write
\begin{equation}
W_-={\bf 1}+\phi(A)\bigl(S(B)-{\bf 1}\bigr)+K,
\end{equation}
therefore $W_-\in\E_{\square}$, and $\pi_{\square}(W_-)=(\Gamma_1,\Gamma_2,\Gamma_3,\Gamma_4)$ is given by
\begin{align*}
\Gamma_1(\beta)&:=S(\beta),\\
\Gamma_2(\alpha)&:=\begin{cases}
1 & \text{if}\ w(0)\neq 0,\\
\tanh(\pi\alpha)+i\cosh(\pi\alpha)^{-1} & \text{if}\ w(0)=0,
\end{cases}\\
\Gamma_3(\beta)&:=1\qquad \text{and} \qquad \Gamma_4(\alpha):=1
\end{align*}
for $\alpha,\beta\in\R$. Here, we have used the low-energy asymptotic of $s(k)$ provided in Section 2.


One can easily observe that $\square$ is homeomorphic to the unit circle $\S^1$. Hence, for a given non-vanishing function $\Gamma\in C(\square)$, the winding number $\wn_{\square}(\Gamma)\in\Z$ of the closed curve $\Gamma(\square)\subset\C\setminus\{0\}$ is well-defined. Here, we turn around $\square$ clockwise as a convention (see Figure \ref{square}). Then, we obtain the following index formula by applying the Gohberg-Krein index theorem:


\begin{theorem}[topological version of Levinson's theorem] \label{topLevinson2}
Under the assumption \eqref{decay assumption} with $\rho>2$, 
\begin{equation}\label{Levinson2}
\wn_{\square}\Bigl(\pi_{\square}(W_-)\Bigr)=-\Index(W_-),
\end{equation}
where $\Index(\cdot)$ denotes the Fredholm index.
\end{theorem}


Theorem \ref{topLevinson2} indeed gives a reformulation of the usual Levinson's theorem \eqref{classical Levinson}. We first observe that 
$$
\Index(W_-)=\dim\ker\bigl(W_{-}\bigr)-\dim\text{coker}\bigl(W_{-}\bigr)=0-\dim\H_{\pp}(H)=-N,
$$
by the asymptotic completeness of $W_-$, and therefore the right hand side in the equality \eqref{Levinson2} is the number of eigenvalues of $H$. On the other hand, by taking $S(\beta)=\exp\bigl(-2i\eta(\e^{\beta})\bigr)$ into account, one computes the contributions $wn(\Gamma_j)$ of $\Gamma_j$ for $j=1,2,3,4$ to the winding number as follows:
\begin{align*}
wn(\Gamma_1)&=\frac{-2\eta(\infty)-\bigl(-2\eta(0)\bigr)}{2\pi}=-\frac{\eta(\infty)-\eta(0)}{\pi},\\
wn(\Gamma_2)&:=\begin{cases}
0 & \text{if}\ w(0)\neq 0,\\
-1/2 & \text{if}\ w(0)=0,
\end{cases}\\
wn(\Gamma_j)&=wn(1)=0\ \text{for}\ j=3,4.
\end{align*}
Here, in the case $w(0)=0$, $wn(\Gamma_2)$ is computed by the formula $\frac{1}{2\pi i}\int_{-\infty}^{\infty}\Gamma_2^*(\alpha)\Gamma_2'(\alpha)\d\alpha$ since $\Gamma_2$ is continuously differentiable. According to the orientation of $\square$, the winding number is given by
\begin{align*}
\wn_{\square}\Bigl(\pi_{\square}(W_-)\Bigr)&=-wn(\Gamma_1)+wn(\Gamma_2)+wn(\Gamma_3)-wn(\Gamma_4)=\begin{cases}
\frac{\eta(\infty)-\eta(0)}{\pi} &\text{if}\ w(0)\neq0,\\
\frac{\eta(\infty)-\eta(0)}{\pi}-\frac{1}{2} &\text{if}\ w(0)=0.
\end{cases}
\end{align*}
Hence, the relation \eqref{classical Levinson} can be obtained by rearranging the equality \eqref{Levinson2}.


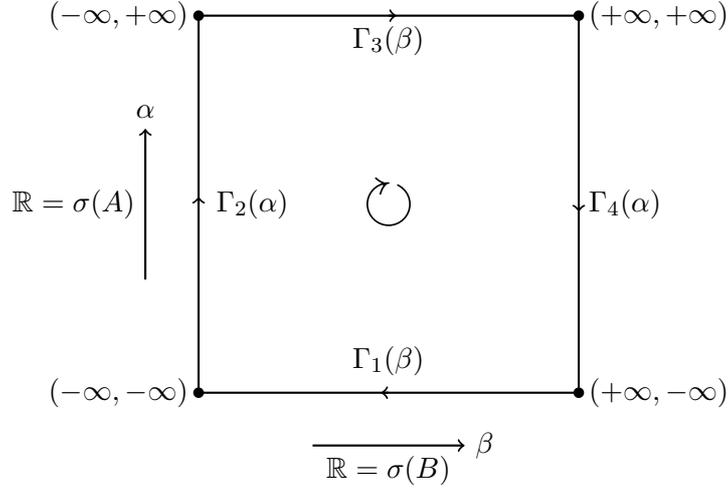
\begin{figure}[htb]
\centering
\begin{tikzpicture}

\draw (0,0) node[left]{$(-\infty,-\infty)$};
\fill (0,0) circle [radius=0.07];

\draw (0,5) node[left]{$(-\infty,+\infty)$};
\fill (0,5) circle [radius=0.07];

\draw (5,0) node[right]{$(+\infty,-\infty)$};
\fill (5,0) circle [radius=0.07];

\draw (5,5) node[right]{$(+\infty,+\infty)$};
\fill (5,5) circle [radius=0.07];

\draw (2.5,-0.7) node[below]{$\R=\sigma(B)$};
\draw (-0.7,2.5) node[left]{$\R=\sigma(A)$};

\draw (2.5,2.5) node{{\Huge$\circlearrowright$}};

\draw[thick, =10pt] (0,5)--(0,2.5);
\draw[thick,<- =10pt] (0,2.6)--(0,0);
\draw[thick, =10pt] (0,0)--(2.5,0);
\draw[thick,<- =10pt] (2.4,0)--(5,0);
\draw[thick, =10pt] (5,0)--(5,2.5);
\draw[thick,<- =10pt] (5,2.4)--(5,5);
\draw[thick, =10pt] (5,5)--(2.5,5);
\draw[thick, <- =10pt] (2.6,5)--(0,5);

\draw (2.5,0.1) node[above]{$\Gamma_1(\beta)$};
\draw (5,2.5) node[right]{$\Gamma_4(\alpha)$};
\draw (2.5,5) node[below]{$\Gamma_3(\beta)$};
\draw (0.1,2.5) node[right]{$\Gamma_2(\alpha)$};

\draw[thick, ->= 10pt] (1.5,-0.7)--(3.5,-0.7) node[right] {$\beta$};
\draw[thick, ->= 10pt] (-0.7,1.5)--(-0.7,3.5) node[above] {$\alpha$};
\end{tikzpicture}

\caption{The boundary $\square$ of $\overline{\R}\times\overline{\R}$ and its orientation.}\label{square}

\end{figure}


\begin{remark}
 One can also directly verify the unitary equivalence between $(B,A)$ and $(X,D)$. We set $\mathscr{R}:=\mathscr{R}_0\F_s: L^2(\R_+)\to L^2(\R)$, where $\mathscr{R}_0:L^2(\R_+)\to L^2(\R)$ is defined by $[\mathscr{R}_0f](\beta):=\e^{\beta/2}f(\e^\beta)$ for $f\in L^2(\R_+)$ and $\beta\in\R$. Then, one can easily see that $\mathscr{R}$ is a unitary operator and satisfies $\mathscr{R}\e^{-itB}=\e^{-itX}\mathscr{R}$ and $\mathscr{R}\e^{-itA}=\e^{-itD}\mathscr{R}$ for any $t\in\R$ by using the relation $\F_s A=-A\F_s$.
\end{remark}

\begin{remark}
Our investigation on the present model is motivated from the well-known fact that our assumption on the potential prevents any singular continuous spectrum and an infinite number of bound states for $H$. However, we can also observe this fact from formula \eqref{formula of the wo}. The Atkinson's characterization states that a bounded operator $a\in B(\H)$ on a Hilbert space $\H$ is Fredholm if and only if $\pi(a)\in B(\H)/K(\H)$ is invertible, where $\pi:B(\H)\to B(\H)/K(\H)$ is the quotient map. Since $\pi(W_-)=\pi_{\square}(W_-)$ is invertible in $C(\square)\cong\E/\K\subset B\left(L^2(\R_+)\right)/\K$, $W_-$ is a Fredholm operator on $L^2(\R_+)$. Hence, $W_-$ can never have infinite dimensional kernel or cokernel. $H$ has therefore no singular continuous spectrum nor infinite number of bound states.
\end{remark}


\begin{remark}
Note that in \cite{Y2}, the low-energy asymptotic of the scattering matrix $s(k)$ for a potential decreasing slowly, i.e. $v(x)\sim v_0 x^{-\rho}$, $\rho\in (0,2)$ as $x\to\infty$, is established. Since the situation is rather complicated for $\rho\in(0,1)$ due to the long-rangeness, let us recall the following result for $\rho\in(1,2)$: the phase shift $\eta$ associated with $v$ satisfies
\begin{equation}\label{asymptotic of eta for long-range}
\eta(k)= \eta_0 k^{1-2/\rho}+O(1)\qquad \text{as}\ k\downarrow 0,
\end{equation}
where $\pm\eta_0>0$ if $\pm v_0>0$. Hence, the scattering matrix $s(k)$ is oscillating near $k=0$.


In \cite{IR,NPR}, it has been proved that the wave operators associated with a specific potential of the form $x^{-2}$ belong to the $C^*$-algebra $\E_\A$ generated by $a(A)b(B)$ with $a\in C(\overline{\R})$ and $b$ belongs to a $C^*$-subalgebra $\A\subset C_{\rm b}(\R)$. $\A$ canonically contains the scattering matrix in the rescaled energy. In the case, when the potential creates the infinitely many bound states, the role of the ideal $\K$ is played by the ideal $\J_\A$ generated by $a(A)b(B)$ with $a\in C_0(\R)$ and $b\in \A$. It is then natural to wonder if we can prove an explicit formula for $W_-$ similar to \eqref{expression of wo} for slowly decreasing potentials with the remainder term belonging to $\J_{\A}$ for some suitable $\A\subset C_{\rm b}(\R)$. This question is at present far from being solved, but the affirmative answer would allow one to obtain a $K$-theoretic index formula by using the boundary map in $K$-theory, as explained in \cite[Sect. 5]{IR} (see also \cite[Sect. 4.2 and 4.3]{R}). In view of \eqref{asymptotic of eta for long-range}, such an algebra $\A$ should contain $C(\overline{\R})$, and any function $f\in C_{\rm b}(\R)$ satisfying that there is a continuous periodic function $f_-$ such that $f(\beta)\sim f_-(\e^{\beta(1-2/\rho)})$ as $\beta\to-\infty$, in a suitable sense.

\end{remark}


\section{Analysis of the remainder term}

In this last section, we prove the compactness of the remainder term $K=F_2\F_s$ in \eqref{formula of the wo}. In the following, we will denote integral operators by the same symbols as their kernels.


By virtue of the integral equation \eqref{Volterra for Jost}, we first observe that 
\begin{align*}
F_2(x,k)&=\sqrt{\frac{2}{\pi}}\frac{1}{2i}\left(p(x,k)s(k)-\overline{p(x,k)}\right)\\
&=\sqrt{\frac{2}{\pi}}\Im\left(p(x,k)\e^{-i\eta(k)}\right)\e^{-i\eta(k)}\\
&=\sqrt{\frac{2}{\pi}}\int_x^{\infty}\frac{\sin \bigl(k(x_1-x)\bigr)}{k}v(x_1)\Im\left(\theta(x_1,k)\e^{-i\eta(k)}\right)\d x_1\times\e^{-i\eta(k)}\\
&=\sqrt{\frac{2}{\pi}}\int_x^{\infty}\frac{\sin \bigl(k(x_1-x)\bigr)}{k}v(x_1)\Im\left(p(x_1,k)\e^{-i\eta(k)}\right)\d x_1\times\e^{-i\eta(k)}\\
&\qquad +\sqrt{\frac{2}{\pi}}\int_x^{\infty}\frac{\sin \bigl(k(x_1-x)\bigr)}{k}v(x_1)\sin\bigl(kx_1-\eta(k)\bigr)\d x_1\times\e^{-i\eta(k)}.
\end{align*}
Therefore, by iterating this procedure one obtains the following decomposition of the kernel $F_2$ for each $N\in\N$:
\begin{equation}\label{decomp of K_2}
F_2(x_0,k)=\left\{p_N(x_0,k)+R_N(x_0,k)\right\}\e^{-i\eta(k)},\qquad \forall x_0,k\in\R_+
\end{equation}
where
\begin{align*}
p_N(x_0,k)=\sqrt{\frac{2}{\pi}}\int_{x_0}^{\infty}\d x_1\cdots\int_{x_{N-1}}^{\infty}\d x_{N}\left(\prod_{j=1}^{N}\frac{\sin \bigl(k(x_j-x_{j-1})\bigr)}{k}v(x_j)\right)\Im\left(p(x_N,k)\e^{-i\eta(k)}\right)
\end{align*}
and
$$
R_N(x_0,k)=\sum_{n=1}^{N}r_n(x_0,k)
$$
with
\begin{align*}
r_n(x_0,k)=\sqrt{\frac{2}{\pi}}\int_{x_0}^{\infty}\d x_1\cdots\int_{x_{n-1}}^{\infty}\d x_{n}\left(\prod_{j=1}^{n}\frac{\sin \bigl(k(x_j-x_{j-1})\bigr)}{k}v(x_j)\right)\sin\bigl(kx_n-\eta(k)\bigr).
\end{align*}
Hence, $K=\left\{p_N\F_s+R_N\F_s\right\}\e^{-i\eta(\sqrt{H_0})}$ by taking into account the intertwining relation between $H_0$ and $\F_s$. 


\begin{lemma}\label{compactness of p_N}
For any $\rho>2$ there exists $N\in\N$ such that $p_N\in L^2(\R_+\times\R_+)$, and accordingly the first term $p_N\F_s\e^{-i\eta(\sqrt{H_0})}$ in the remainder term $K$ is Hilbert-Schmidt on $L^2(\R_+)$.
\end{lemma}


\begin{proof}
By the assumption \eqref{decay assumption}, one can prove that there exists $C_0>0$ such that for $x_N\geq\ldots\geq x_1\geq x_0\geq 0$
\begin{align}\label{messy estimate}
\prod_{j=1}^N\left| \frac{\sin\bigl(k(x_j-x_{j-1})\bigr)}{k}v(x_j)\right|\leq C_0\begin{cases}
 k^{-N}\prod_{j=1}^N(1+x_j)^{-\rho} & \text{if}\ k>1,\\
 \prod_{j=1}^N(1+x_j)^{-(\rho-1)} & \text{if}\ 0\leq k\leq 1.
\end{cases}
\end{align}
By the estimates \eqref{estimate of K2} and \eqref{messy estimate}, one has
\begin{align*}
|p_N(x_0,k)|&\leq\int_{x_0}^{\infty}\d x_1\cdots\int_{x_{N-1}}^{\infty}\d x_{N}\left|\prod_{j=1}^{N}\frac{\sin \bigl(k(x_j-x_{j-1})\bigr)}{k}v(x_j)\right|\times\sqrt{\frac{2}{\pi}}|p(x_N,k)|\\
&\leq C\int_{x_0}^{\infty}\d x_1\cdots\int_{x_{N-1}}^{\infty}\d x_{N}\begin{cases}
 k^{-N}\prod_{j=1}^N(1+x_j)^{-\rho} \times k^{-1}(1+x_N)^{-(\rho-1)} & \text{if}\ k>1,\\
 \prod_{j=1}^N(1+x_j)^{-(\rho-1)} \times (1+x_N)^{-(\rho-2)} & \text{if}\ 0\leq k\leq 1.
\end{cases}
\end{align*}
Finally, by direct computations one obtains 
\begin{align}\label{estimate of p_N}
|p_N(x,k)|\leq C\begin{cases}
k^{-(N+1)}(1+x)^{-(N+1)(\rho-1)} & \text{if}\ k>1,\\
(1+x)^{-(N+1)(\rho-2)} & \text{if}\ 0\leq k\leq 1,
\end{cases}
\end{align}
for any $x\geq 0$. Hence, if we take $N\in\N$ so that $2(N+1)(\rho-2)>1$, then $p_N\in L^2(\R_+\times\R_+)$.
\end{proof}


$p_N$ is a good part in the remainder term while we need a special treatment for $R_N$.  Indeed, since we need to deal with the kernel $r_n$ for small $n$, the above argument does not work for the second term $R_N\F_s\e^{-i\eta\left(\sqrt{H_0}\right)}$, and therefore the operator $r_n$ should be understood as a kind of oscillatory integral operator. 


Our approach is based on the following observation: if we set $V_v(x):=\int_x^{\infty}v(y)\d y$, then $V_v$ is differentiable and has the derivative equal to $-v$ almost everywhere, and  by an integration by parts and a change of variables one obtains 
\begin{align}\label{observation}
r_1(x,k)&=\sqrt{\frac{2}{\pi}}\int_{x}^{\infty}V_v(y)\sin\bigl(k(2y-x)-\eta(k)\bigr)\d y=\frac{1}{2}\sqrt{\frac{2}{\pi}}\int_{x}^{\infty}V_v\left(\frac{x+y}{2}\right)\sin \bigl(ky-\eta(k)\bigr)\d y.
\end{align}
Moreover, by using Fubini theorem one can see that $r_1$ equals the product of an integral operator with the kernel $\frac{1}{2}V_v\left(\frac{x+y}{2}\right)Y(y-x)$, which belongs to $L^2(\R_+\times\R_+)$, and a bounded integral operator with the kernel $\sin (ky-\eta(k))$. Here, $Y$ denotes the Heaviside function on $\R$.


The singular part $k^{-1}$ in the kernel $r_n(x,k)$ might be avoided by iterating similar computation with several integration by parts, and one might obtain the compactness of $r_n$. However, one gets a much more complicated dependence on the potential $v$ in the formula of $r_n$ for $n=2,3,\ldots$, after such computations (see Remark \ref{final remark} at the end of this section). 


In order to deal with this, we introduce an algebraic framework for integral operators of the type $r_n$. For $p=1,2$, let us consider the complex vector space
$$
\V_{p}:=\bigcup_{\varepsilon>0}L^{\infty}_{p+\varepsilon}(\R_+) \qquad \text{with} \qquad L^{\infty}_{p+\varepsilon}(\R_+):=\left\{V\in L^{\infty}(\R_+) \mid  \|(1+\cdot)^{p+\varepsilon}V(\cdot)\|_{L^{\infty}(\R_+)}<\infty \right\}.
$$
For $V_j\in\V_{1}$, $j=1,2$ we set
$$
V_1\star V_2(x):=\int_{x}^{\infty}V_1(y)V_2(y)\d y,\qquad x\in\R_+.
$$
One can easily check that $V_1\star V_2\in\V_{1}$. As a consequence, $\V_1$ becomes a non-associative commutative $\C$-algebra under $\star$. Furthermore, motivated from the computations for $r_1$, we define $V_u(x):=\int_{x}^{\infty}u(y)\d y$ for any $u\in\V_2$. Then, $V_u\in\V_{1}$ and it has the derivative equal to $-u$ almost everywhere in $\R_+$.


\begin{lemma}\label{step1}
Let $n\in\N$. For $V_1,\ldots,V_n\in\V_1$, let $W[V_1,\ldots,V_n]$ be the integral operator with the kernel
\begin{align*}
W[V_1,\ldots,V_n](x_0,k)&:=\sqrt{\frac{2}{\pi}}\int_{x_0}^{\infty}\d x_1\cdots\int_{x_{n-1}}^{\infty}\d x_{n}\left(\prod_{j=1}^{n}V_{j}(x_j)\right)\\
&\qquad\qquad\times\sin\left(k\left(2\sum_{\ell=0}^{n-1}(-1)^{\ell}x_{n-\ell}+(-1)^{n}x_0\right)-\eta(k)\right).
\end{align*}
Then, the operator $W[V_1,\ldots,V_n]$ extends continuously to a Hilbert-Schmidt operator on $L^2(\R_+)$.
\end{lemma}


\begin{proof}
Note first that by a change of variables, one can prove that for any $x_0,k\in\R_+$
\begin{align*}
&2^n\sqrt{\frac{\pi}{2}}\times W[V_1,\ldots,V_n](x_0,k)\\
&=\begin{cases}
\displaystyle \int_{x_0}^{\infty} V_1\left(\frac{x_1+x_0}{2}\right)\sin\bigl(kx_1-\eta(k)\bigr) \d x_1 & \text{if}\ n=1,\\
\displaystyle \int_{x_0}^{\infty} V_{1}\left(\frac{x_1+x_{0}}{2}\right)\int_{x_0}^{\infty} V_{2}\Bigl(\frac{x_2+x_{1}}{2}\Bigr)\sin\bigl(kx_2-\eta(k)\bigr) \d x_2\d x_1 &\text{if}\ n=2,\\
\displaystyle \int_{x_0}^{\infty}\d x_1\int_{x_0}^{\infty}\d x_2\int_{x_1}^{\infty}\d x_3\cdots\int_{x_{n-2}}^{\infty}\d x_{n}\left(\prod_{j=1}^{n}V_{j}\left(\frac{x_j+x_{j-1}}{2}\right)\right)\sin\bigl(kx_n-\eta(k)\bigr) & \text{if}\ n>2.
\end{cases}
\end{align*}
Let us prove this in the case $n=3$ for the convenience of the reader. We do changes of the variables first as $y_1=2x_1-x_0$, second $y_2=2x_2-y_1$ and finally $y_3=2x_3-y_2$:
\begin{align*}
&W[V_1,V_2,V_3](x_0,k)\\
&=\sqrt{\frac{2}{\pi}}\int_{x_0}^{\infty}\d x_1\int_{x_1}^{\infty}\d x_2\int_{x_{2}}^{\infty}\d x_{3}\left(\prod_{j=1}^{3}V_{j}(x_j)\right)\sin\Bigl(k\left(2x_3-2x_2+2x_1-x_0\right)-\eta(k)\Bigr)\\
&=\sqrt{\frac{2}{\pi}}\int_{x_0}^{\infty}\frac{\d y_1}{2}\int_{\frac{y_1+x_0}{2}}^{\infty}\d x_2\int_{x_{2}}^{\infty}\d x_{3}V_1\left(\frac{y_1+x_0}{2}\right)\left(\prod_{j=1}^{2}V_{j}(x_j)\right)\sin\Bigl(k\left(2x_3-2x_2+y_1\right)-\eta(k)\Bigl)\\
&=\sqrt{\frac{2}{\pi}}\int_{x_0}^{\infty}\frac{\d y_1}{2}\int_{x_0}^{\infty} \frac{\d y_2}{2}\int_{\frac{y_2+y_1}{2}}^{\infty}\d x_{3}V_1\left(\frac{y_1+x_0}{2}\right)V_2\left(\frac{y_2+y_1}{2}\right)V_3(x_3)\sin\Bigl(k\left(2x_3-y_2\right)-\eta(k)\Bigr)\\
&=\sqrt{\frac{2}{\pi}}\int_{x_0}^{\infty}\frac{\d y_1}{2}\int_{x_0}^{\infty} \frac{\d y_2}{2}\int_{y_1}^{\infty}\frac{\d y_{3}}{2}V_1\left(\frac{y_1+x_0}{2}\right)V_2\left(\frac{y_2+y_1}{2}\right)V_3\left(\frac{y_2+y_3}{2}\right)\sin\bigl( k y_3-\eta(k)\bigr).
\end{align*}


We then define the kernel $(x,y)\mapsto U[V_1,\ldots,V_n](x,y)$ by
\begin{align*}
&2^{n}\times U[V_1,\ldots,V_n](x_0,x_n)\\
&=\begin{cases}
\displaystyle V_1\left(\frac{x_1+x_0}{2}\right)Y(x_1-x_0)& \text{if}\ n=1,\\
\displaystyle\int_{x_0}^{\infty} V_{1}\left(\frac{x_1+x_{0}}{2}\right)V_{2}\left(\frac{x_2+x_{1}}{2}\right)\d x_1 Y(x_2-x_0)& \text{if}\ n=2,\\
\displaystyle\int_{x_0}^{\infty}\d x_1\int_{x_0}^{\infty}\d x_2\int_{x_1}^{\infty}\d x_3\cdots\int_{x_{n-3}}^{\infty}\d x_{n-1}\left(\prod_{j=1}^{n}V_{j}\left(\frac{x_j+x_{j-1}}{2}\right)\right)Y(x_{n}-x_{n-2}) &\text{if}\ n>2.
\end{cases}
\end{align*}
We also set $\Phi:=\cos\left(\eta\left(\sqrt{H_0}\right)\right)-\F_c\F_s\sin\left(\eta\left(\sqrt{H_0}\right)\right)$.  Then, by using the addition formula for sine and the intertwining relation between $\F_s$ and $H_0$, one obtains
\begin{equation}\label{decomposition of W}
W[V_1,\ldots,W_n]=U[V_1,\ldots,V_n]\Phi\F_s.
\end{equation}


To obtain an estimate for $U[V_1,\ldots,V_n]$, we take $\varepsilon>0$ sufficiently small such that $V_{j}\in L^{\infty}_{1+\varepsilon}(\R_+)$ for all $j=1,\ldots,n$. Then one has for instance
\begin{align}\label{estimate of U}
\bigl|U[V_1,\ldots,V_n](x_0,x_n)\bigr|\leq C(1+x_0)^{-\frac{1+\varepsilon}{2}-(n-1)\varepsilon}(1+x_n)^{-\frac{1+\varepsilon}{2}} Y(x_n-x_0).
\end{align}
Here, we first applied the AM-GM inequality to the function $1+\frac{x_j+x_{j-1}}{2}$, and then replaced the lower ends of the integrals and $x_{n-2}$ in $Y(x_n-x_{n-2})$ by $x_0$.  Hence, $U[V_1,\ldots,V_n]\in L^{2}(\R_+\times\R_+)$. Since $\Phi$ is a bounded operator on $L^2(\R_+)$, $W[V_1,\ldots,V_n]$ is a Hilbert-Schmidt operator.
\end{proof}


We define $\W$ and $\U$ to be the (algebraic) linear span of the kernels $W[V_1,\ldots,V_n]$ and $U[V_1,\ldots,V_n]$ for any $V_1,\ldots,V_n\in\V_1$ and any $n\in\N$, respectively. Then, according to Lemma \ref{step1} both $\W$ and $\U$ are vector subspaces of $L^2(\R_+\times\R_+)$. Note that by identifying the set of kernels and the set of integral operators on $L^2(\R_+)$ the equality \eqref{decomposition of W} can be rewritten as
\begin{equation}\label{decomposition of W2}
\W=\U\Phi\F_s.
\end{equation}
It follows from the estimate \eqref{estimate of U} that  for any element $U\in\U$ there exists some $\varepsilon>0$ such that the estimate
\begin{equation}\label{estimate of U2}
\bigl|U(x,y)\bigr|\leq C(1+x)^{-\frac{1+\varepsilon}{2}}(1+y)^{-\frac{1+\varepsilon}{2}}Y(y-x)
\end{equation}
holds.  


\begin{lemma}\label{step2}
For $W\in\W$ and $u\in \V_2$, set 
\begin{equation}\label{bilinear map}
[\mathfrak{S}_uW](x,k):=\int_{x}^{\infty}\frac{\sin \bigl(k(y-x)\bigr)}{k}u(y)W(y,k)\d y,\qquad x,k\in\R_+.
\end{equation}
Then, $\mathfrak{S}_u$ is a well-defined linear map on $\W$ for any $u\in\V_2$.
\end{lemma}


\begin{proof}
The linearity is obvious. Hence, it suffices to prove the well-definedness for generators of the form $W[V_1,\ldots,V_n]$, and therefore we prove it by induction for $n$. Let $V\in\V_1$. Then, by an integration by parts with respect to $x_1$ one has for any $x_0,k\in\R_+$,
\begin{align*}
&\Bigl[\mathfrak{S}_uW[V]\Bigr](x_0,k)= \int_{x_0}^{\infty}\frac{\sin \bigl(k(x_1-x_0)\bigr)}{k}u(x_1)W[V](x_1,k)\d x_1\\
&=\sqrt{\frac{2}{\pi}}\int_{x_0}^{\infty}\frac{\sin \bigl(k(x_1-x_0)\bigr)}{k}u(x_1)\int_{x_1}^{\infty}V(x_2)\sin\Bigl( k(2x_2-x_1)-\eta(k)\Bigr)\d x_2\d x_1\\
&=\sqrt{\frac{2}{\pi}}\int_{x_0}^{\infty}\cos \bigl(k(x_1-x_0)\bigr)V_u(x_1)\int_{x_1}^{\infty}V(x_2)\sin\Bigl( k(2x_2-x_1)-\eta(k)\Bigr)\d x_2\d x_1\\
&\qquad -\sqrt{\frac{2}{\pi}}\int_{x_0}^{\infty}\sin \bigl(k(x_1-x_0)\bigr)V_u(x_1)\int_{x_1}^{\infty}V(x_2)\cos\Bigl( k(2x_2-x_1)-\eta(k)\Bigr)\d x_2\d x_1\\
&\qquad -\sqrt{\frac{2}{\pi}}\int_{x_0}^{\infty}\frac{\sin \bigl(k(x_1-x_0)\bigr)}{k}V_u(x_1)V(x_1)\sin\bigl( kx_1-\eta(k)\bigr)\d x_1\\
&=\sqrt{\frac{2}{\pi}}\int_{x_0}^{\infty}V_u(x_1)\int_{x_1}^{\infty}V(x_2)\sin\Bigl( k(2x_2-2x_1+x_0)-\eta(k)\Bigr)\d x_2\d x_1 \\
&\qquad -\sqrt{\frac{2}{\pi}}\int_{x_0}^{\infty}\frac{\sin \bigl(k(x_1-x_0)\bigr)}{k}V_u(x_1)V(x_1)\sin\bigl( kx_1-\eta(k)\bigr)\d x_1.
\end{align*}
The first term is equal to $W[V_u,V](x_0,k)$, and again by one integration by parts 
\begin{align*}
&\sqrt{\frac{2}{\pi}}\int_{x_0}^{\infty}\frac{\sin \bigl(k(x_1-x_0)\bigr)}{k}V_u(x_1)V(x_1)\sin\bigl( kx_1-\eta(k)\bigr)\d x_1=W[V_u\star V](x_0,k).
\end{align*}
Therefore, one has 
\begin{equation}\label{for 1}
\mathfrak{S}_uW[V]=W[V_u,V]-W[V_u\star V]\in\W.
\end{equation}


Now, we assume that the statement is true for $n$ and for any $u\in\V_2$. Let $V_1,\ldots,V_{n+1}\in\V_1$. We shall prove the equality
\begin{equation}\label{for n+1}
\mathfrak{S}_uW[V_1,\ldots,V_{n+1}]=W[V_u,V_1,\ldots,V_{n+1}]-\mathfrak{S}_{V_uV_1}W[V_2,\ldots,V_{n+1}].
\end{equation}
Then, $\mathfrak{S}_uW[V_1,\ldots,V_{n+1}]\in\W$ since $V_uV_1\in \V_{2}$. 


Let us prove \eqref{for n+1}. By an integration by parts, one has 
\begin{align*}
&\Bigl[\mathfrak{S}_{u}W[V_1,\ldots,V_{n+1}]\Bigr](x_0,k)=\int_{x_0}^{\infty}\frac{\sin \bigl(k(x_1-x_0)\bigr)}{k}u(x_1)W[V_1,\ldots,V_{n+1}](x_1,k)\d x_1\\
&=\int_{x_0}^{\infty}\cos\bigl(k(x_1-x_0)\bigr)V_u(x_1)W[V_1,\ldots,V_{n+1}](x_1,k)\d x_1\\ 
& \qquad +\int_{x_0}^{\infty}\frac{\sin \bigl(k(x_1-x_0)\bigr)}{k}V_u(x_1)\Bigl(\partial_{x_1}W[V_1,\ldots,V_{n+1}]\Bigr)(x_1,k)\d x_1\\
&=:Z_1(x_0,k)+Z_2(x_0,k).
\end{align*}
The derivative $\Bigl(\partial_{x_1}W[V_1,\ldots,V_{n+1}]\Bigr)(x_1,k)$ in the integrand of $Z_2(x_0,k)$ is equal to the sum of $-V_{1}(x_1)W[V_2,\ldots,V_{n+1}](x_1,k)$, which gives the second term in \eqref{for n+1}, and 
\begin{align*}
w(x_1,k):=&k\sqrt{\frac{2}{\pi}}\int_{x_1}^{\infty}\d x_2\cdots\int_{x_{n+1}}^{\infty}\d x_{n+2}\left(\prod_{j=1}^{n+1}V_{j}(x_{j+1})\right)\\
& \qquad \times(-1)^{n+1}\cos\left(k\left(2\sum_{\ell=0}^{n}(-1)^{\ell}x_{n+2-\ell}+(-1)^{n+1}x_1\right)-\eta(k)\right).
\end{align*}


Now, let us have a look at the trigonometric functions in  the integrand of $Z_1(x_0,k)$ and the expression $k^{-1}\sin\left(k(x_1-x_0)\right)w(x_1,k)$. We have
\begin{align*}
&\cos\bigl(k(x_1-x_0)\bigr)\sin\left(k\left(2\sum_{\ell=0}^{n}(-1)^{\ell}x_{n+2-\ell}+(-1)^{n+1}x_1\right)-\eta(k)\right)\\
&\ \ \ +(-1)^{n+1}\sin\bigl(k(x_1-x_0)\bigr)\cos\left(k\left(2\sum_{\ell=0}^{n}(-1)^{\ell}x_{n+2-\ell}+(-1)^{n+1}x_1\right)-\eta(k)\right)\\
&=\sin\left(k\left(2\sum_{\ell=0}^{n}(-1)^{\ell}x_{n+2-\ell}+(-1)^{n+1}x_1\right)-\eta(k)+(-1)^{n+1}k(x_1-x_0)\right)\\
&=\sin\left(k\left(2\sum_{\ell=0}^{n+1}(-1)^{\ell}x_{n+2-\ell}+(-1)^{n+2}x_0\right)-\eta(k)\right).
\end{align*}
Hence, we have
\begin{align*}
Z_1(x_0,k)+\int_{x_0}^{\infty}\frac{\sin \bigl(k(x_1-x_0)\bigr)}{k}V_u(x_1)w(x_1,k)\d x_1=W[V_u,V_1,\ldots,V_{n+1}](x_0,k).
\end{align*}
This finishes the proof of \eqref{for n+1}

\end{proof}


\begin{proof}[Proof of Theorem \ref{expression of wo}]
It suffices to prove that the remainder term $K=F_2\F_s$ is Hilbert-Schmidt. By \eqref{observation} one has $r_1=W[V_v]\in\W$. One can also see from the definition that $r_{n+1}=\mathfrak{S}_vr_n$. Therefore, it follows from Lemma \ref{step2} that $r_n$ belongs to $\W$ for any $n\in\N$. By Lemma \ref{step1}, $R_N=\sum_{n=1}^N r_n$ is a Hilbert-Schmidt operator. Hence, together with Lemma \ref{compactness of p_N} we finish the proof.
\end{proof}


\begin{remark}\label{final remark}
One can easily see from Lemma \ref{step2} that \eqref{bilinear map} defines a bilinear map $\mathfrak{S}:\V_2\times\W\to\W$. Moreover, by combining \eqref{for 1} and \eqref{for n+1} and by induction, one can also prove the equality 
\begin{align}\label{for n}\notag
\mathfrak{S}_uW[V_1,\ldots,V_n]&=W[V_u,V_1,\ldots,V_{n}]+(-1)^1W[V_u\star V_1,V_2,\ldots,V_n]+\cdots\\ \notag
&\qquad +(-1)^{n-1}W\Bigl[\Bigl(\cdots\Bigl(\bigl(V_u\star V_1\bigr)\star V_2\Bigr)\star\cdots\Bigr)\star V_{n-1},V_{n}\Bigr]\\
&\qquad \qquad+(-1)^{n}W\Bigl[\Bigl(\Bigl(\cdots\Bigl(\bigl(V_u\star V_1\bigr)\star V_2\Bigr)\star\cdots\Bigr)\star V_{n-1}\Bigl)\star V_{n}\Bigr]
\end{align}
for any $n$. Note that parentheses in the right hand side of \eqref{for n} are indispensable since the product $\star$ is not associative. This property of $\mathfrak{S}$ is the source of the complicated dependence of $r_n=\mathfrak{S}_v^{n-1}W[V_v]$ on the potential $v$.
\end{remark} 


\end{document}